\documentclass[12pt]{amsart}

\usepackage[english]{babel}
\usepackage[utf8x]{inputenc}
\usepackage{amsmath, amstext, amssymb, amsthm, amscd, graphicx, dsfont, manfnt, mathrsfs, marvosym,bm,fullpage}
\usepackage[integrals]{wasysym}
\newtheorem{thm}{Theorem}
\newtheorem{defi}{Definition}
\newtheorem{nota}{Notation}

\newtheorem{lem}{Lemma}

\newtheorem{cor}{Corollary}
\newtheorem*{rem}{Remark}
\newcommand{\intn}{\int_{\mathcal{N}_q}}

\newcommand{\qq}{\mathbb{Q}(q)}
\newcommand{\vv}{\mathsf{V}}
\newcommand{\Eta}{\mathrm{H}}
\newcommand{\dgue}{\textrm{-GUE}}

\newcommand{\Asym}{\operatorname{Asym}}

\title{On $q$-analogs of some integrals over GUE}
\author{Praveen S. Venkataramana}

\begin{document}

\begin{abstract}
Statistics over the Gaussian unitary ensemble and the Wishart ensemble of random matrices often have nice closed-form expressions. These are related to multivariate extensions of the Hermite, Laguerre, and Jacobi polynomials, which often occur in the study of these ensembles.

In the paper, we develop a formal $q$-analog of the Gaussian unitary ensemble, using $q$-Hermite polynomials and coefficient extraction instead of integration. This way we derive $q$-analogs for many well-known eigenvalue statistics. One of these is related to the Harer-Zagier formula, which uses a matrix integral to count the number of unicellular maps on $n$ vertices by genus.
\end{abstract}

\maketitle

\section{Introduction}

The classical orthogonal polynomials play a vital role in both pure and applied mathematics. Recent developments in random matrix theory, combinatorics and mathematical physics have motivated a study of analogous polynomials in more than one variable.

The multivariate Hermite, Laguerre, and Jacobi polynomials have been fairly well studied in the last several decades. The Schur polynomials, which generalize the monomials, have a much deeper history, going back to Cauchy. These polynomials are related to the Gaussian unitary ensemble and the Wishart ensemble, and we often get simple closed-form expressions for statistics of these ensembles. These expressions have been analyzed in terms of Cherednik operators, which are special differential-difference operators on symmetric functions.

We believe that these simple expressions show up because the classical polynomials are special cases of hypergeometric functions. So a direction for further research would be to look at the Askey scheme of hypergeometric orthogonal polynomials. Koekoek and Swarttouw showed that these polynomials are simply limits of Askey-Wilson polynomials.

In this paper, we begin exploring this subject by looking at $q$-analogs of formulas involving the most famous and highly structured  random matrix ensemble, the Gaussian Unitary Ensemble, which is remarkable in part because it is possible to represent exact moment formulas so readily. The key univariate idea is that integration with respect to the normal distribution is a linear operator that is naturally orthogonalized in terms of the Hermite polynomials. In this paper, I extend this idea to $q$-Hermite polynomials. We can define $q$-analogs of the  multivariate Hermite polynomials, and indeed any orthogonal polynomials, through a determinant construction reminiscent of the definition of Schur polynomials.  Upon doing so we derive identities for moments of the $q$-GUE.

Curiously, there are problems in map enumeration that are solved by integration over ensembles of tridiagonal matrices that generalize the Gaussian unitary ensemble. For example, in 1986, Harer and Zagier derived a formula for counting one-face maps, which essentially amounts to computing moments of power sums over the GUE [1]. Theorem 5 in this paper is a similar closed-form expression for moments  of power sums over the $q$-GUE. We believe that these results can be refined by substituting the Hermite polynomials with more general orthogonal polynomials, such as the $q$-analogs and the Askey scheme. For one thing, we believe that the Harer-Zagier formula for the number of number of unicellular maps has an analog in this general setting.

\section{Definitions}

In this paper, we mainly work over extensions of the field $\mathbb{Q}(q)$. We define two $q$-analogs of the integers:

$$[n]_q = 1+q+q^2+\cdots+q^{n-1}$$

$$[n]_{q^2} = 1+q^2+q^4+\cdots+q^{2n-2}$$

There are also natural equivalents of the factorial function and the binomial coefficients:

$$[n]^!_q = \prod_{i=1}^n [i]_q$$

$$[n]^!_{q^2} = \prod_{i=1}^n [i]_{q^2}$$

$$\left[{n\atop k}\right]_q = \frac{[n]^!_q}{[k]^!_q[n-k]^!_q}$$

$$\left[{n\atop k}\right]_{q^2} = \frac{[n]^!_{q^2}}{[k]^!_{q^2}[n-k]^!_{q^2}}$$

and four important formal power series:

$$e(x,q) = \sum_{n\ge 0} \frac{x^n}{[n]^!_q}$$
$$e(x,q^2) = \sum_{n\ge 0} \frac{x^n}{[n]^!_{q^2}}$$

$$E(x,q) = \sum_{n\ge 0} \frac{x^nq^{n(n-1)/2}}{[n]^!_q}$$
$$E(x,q^2) = \sum_{n\ge 0} \frac{x^nq^{n(n-1)}}{[n]^!_{q^2}}$$

It follows from the well-known $q$-binomial theorem that: $$e(x,q)E(-x,q) = e(x,q^2)E(-x,q^2) = 1$$

\section{Hermite and Shadow Hermite Polynomials}
The Hermite polynomials often show up when working with the GUE, and there are many ways to define it. For example, in quantum physics, particularly in the analysis of the Calogero-Sutherland model [4], they show up as eigenvectors of certain second-order differential operators. In this paper, we approach them purely formally, using operator calculus. We define the $n$th Hermite polynomial $H_n(x;q)$ of \textit{scalar} argument as follows:
\begin{defi}
 $$H_n(x;q) = E(-D_q^2/(1+q),q^2)x^n$$
 \end{defi}

When $q = 1$, we have $E(-D_q^2/(1+q),q^2) = e^{-\frac{1}{2} \frac{d^2}{dx^2}}$, so that $H_n(x;1)$ are the classical Hermite polynomials of scalar argument. We now define these same Hermite polynomials coefficientwise, as in [1]:

\begin{lem} Let $M_q(n) = [n]_q[n-2]_q[n-4]_q\dots$. Then:
$$H_n(x;q)=\sum_{k=0}^{n/2} (-1)^k (q^2)^{k(k-1)/2} \left[{n\atop 2k}\right]_qM_q(2k-1)x^{n-2k}$$
\end{lem}
\begin{proof}
Observe that:

$[k]_{q^2} = (1-q^{2k})/(1-q^2) = [2k]_q/[2]_q$ so that:

$$(D_q^{2k})x^n/[k]_{q^2}^! = \frac{[n]_q[n-1]_q\cdots[n-2k+1]_q}{[k]_{q^2}[k-1]_{q^2}\cdots[1]_{q^2}}x^{n-2k}$$

$$= \frac{[n]_q^!(1+q)^k}{[n-2k]_q^![2k]_q[2k-2]_q...[2]_q}x^{n-2k} = \frac{[n]_q^!M_q(2k-1)(1+q)^k}{[n-2k]_q^![2k]_q^!)}x^{n-2k}$$

$$ = \left[{n\atop 2k}\right]_qM_q(2k-1)(1+q)^kx^{n-2k}$$

Hence:

\begin{align*}\sum_{k=0}^{n/2} (-1)^k (q^2)^{k(k-1)/2} \left[{n\atop 2k}\right]_qM_q(2k-1)x^{n-2k} &= \left( \sum_{k=0}^{n/2} (-1)^k (q^2)^{k(k-1)/2} \frac{(D_q^2/(1+q))^k}{[k]_{q^2}^!}\right) x^n\\ &= E(-D_q^2/(1+q),q^2)x^n\end{align*}
\end{proof}

Since the operator $E(-D_q^2/(1+q),q^2)$ is a polynomial in $D_q$, it commutes with $D_q$, and we can see that:
\[D_q H_n(x;q) = E(-D_q^2/(1+q),q^2) D_q x^n = E(-D_q^2/(1+q),q^2)[n]_q x^{n-1} = [n]_q H_{n-1}(x;q)\]

Lemma 4.11 in [1] states that the Hermite polynomials satisfy a three-term recurrence: $$H_{n+1} = xH_n - q^{n-1}[n]_qH_{n-1}$$ By Favard's theorem, they are orthogonal with respect to an inner product $\langle\cdot,\cdot\rangle_q$ in which the multiplication map $f(x,q)\mapsto xf(x,q)$ is self-adjoint. In particular, we have the following:

\begin{lem}
Define a linear functional $L : (\mathbb{Q}(q))[x]\to \mathbb{Q}(q)$ by $L(H_n) = 1$ if $n = 0$, and $L(H_n)=0$ for all other $n$. Then:
$$L(H_nH_m) = q^{\binom{n}{2}}[n]_q^!\delta_{n,m}$$
\end{lem}
\begin{proof}
Let $F$ be the algebraic closure of $\mathbb{Q}(q)$, and choose square roots of $q$ and $[n]_q$, for every $n$. Let $h_n(x;q)$ be polynomials in $x$ with coefficients in $F$ defined by $h_0=1$ and:
$$xh_n(x) = \sqrt{q^{n-1}[n]_q}h_{n-1}(x) + \sqrt{q^{n}[n+1]_q}h_{n+1}(x)$$
Define a bilinear form $B : F[x]\times F[x] \to F$ by $B(h_n,h_m)=\delta_{n,m}$. Then it's clear that $B(p,q)=B(q,p)$ for any two polynomials $p,q\in F[x]$, and:
$$B(xh_n,h_m) = \sqrt{q^{n-1}[n]_q}B(h_{n-1},h_m)+\sqrt{q^{n}[n+1]_q}B(h_{n+1},h_m)$$
We claim that this is equal to $B(h_n,xh_m)$. Both these expressions vanish unless $|n-m|=1$. If $n+1=m$,
$$B(xh_n,h_m)=\sqrt{q^{n}[n+1]_q}=\sqrt{q^{m-1}[m]_q}=B(h_n,xh_m)$$
If $n-1=m$, the same argument works, interchanging $m$ and $n$. By linearity, $B(pq,r)=B(q,pr)$ for any three polynomials $p,q,r\in F[x]$. In particular, $B(h_mh_n,1)=B(h_m,h_n)$ Observe that if $\tilde{h}_n(x;q)=\sqrt{q^{\binom{n}{2}}[n]_q^!}h_n(q)$, then $\tilde{h}_0 = 1 = h_0$, and $$ \tilde h_{n+1} = x \tilde h_n - q^{n-1}[n]_q \tilde h_{n-1}$$ so $H_n = \tilde h_n$ for all $n$. Therefore:
$$B(H_mH_n,1)=q^{\binom{n}{2}}[n]_q^!\delta_{n,m}$$
Setting $n = 0$, we see that $B(H_m,1) = L(H_m)$. By linearity, $B(p,1)=L(p)$ for all $p$ in $F[x]$, so \textit{a fortiori}, $$L(H_mH_n)=B(H_mH_n,1)=q^{\binom{n}{2}}[n]_q^!\delta_{n,m}\textrm{ .}$$

\end{proof}

Thus we can formally define a $q$-analog of integrating a polynomial with respect to the normal distribution as follows:

\begin{nota}
$$\int_{\mathcal{N}_q}p(x)\,dx = L(p)$$
\end{nota}
Keep in mind that $\mathcal{N}_q$ is not a set, but shorthand for a hypothetical $q$-analog of the normal distribution (which is of course equivalent to the standard normal distribution when $q=1$). 

\subsection{The Shadow-Hermite Polynomials}

There is a more concrete way to evaluate $\int_{\mathcal{N}_q}p(x)\,dx$:

\begin{lem}
Let $p\in (\mathbb{Q}(q))[x]$. Then $\int_{\mathcal{N}_q} p(x)H_k(x;q)\,dx$ is $q^{\binom{k}{2}}[k]_q^!$ times the coefficient of $x^k$ in the polynomial $ e(D_q^2/(1+q),q^2)p$.
\end{lem}
\begin{proof} We expand $p$ in Hermite polynomials:
$$p(x) = \sum a_n H_n(x;q)$$
Thus:
\begin{align*}e(D_q^2/(1+q),q^2)p &  = \sum a_n e(D_q^2/(1+q),q^2)H_n\\ &=\sum a_ne(D_q^2/(1+q),q^2)E(-D_q^2/(1+q),q^2) x^n \\&= \sum a_n x^n\end{align*}\qedhere
\end{proof}
The coefficient of $x^k$ in this polynomial is $a_k$. However,
$$\intn p(x)H_k(x;q)\,dx = \sum a_n \intn H_n(x;q)H_k(x;q)\,dx = \sum a_n\delta_{n,k}q^{\binom{n}{2}}[n]_q^!=a_kq^{\binom{k}{2}}[n]_k^!$$

By setting $k=0$, we obtain the following corollary:

\begin{cor}
Let $p\in (\mathbb{Q}(q))[x]$. Then $\int_{\mathcal{N}_q} p(x)\,dx$ is the constant term of the polynomial $ e(D_q^2/(1+q),q^2)p$.
\end{cor}

Thus we can define the \textit{shadow Hermite polynomials of scalar argument} $S_n(x;q)$ by:
$$S_n(x;q)=e(D_q^2/(1+q)x^n$$ We then have the following:
\[S_n(x;q) = \sum_{k=0}^{n/2} \left[{n\atop 2k}\right]_q M_q(2k-1) x^{n-2k}\]
and when $n$ is even,
\[\intn x^n\,dx = S_n(0;q) = M_q(n-1)\]
Since the operator $e(D_q^2/(1+q),q^2)$ is a polynomial in $D_q$, it commutes with $D_q$, and we can see that:
\[D_q S_n(x;q) = e(D_q^2/(1+q),q^2) D_q x^n = e(D_q^2/(1+q),q^2)[n]_q x^{n-1} = [n]_q S_{n-1}(x;q)\]

If $\ell>0$ and $N>0$, we can also define polynomials that look like the shadow Hermite polynomials but are missing the first $N$ coefficients:
\[T_{N,\ell}(x;q)=\sum_{k=0}^{\ell/2} \left[{N+\ell\atop 2k}\right]_q M_q(2k-1) x^{N+\ell-2k}\]
The following elementary properties hold directly from the definition:

\begin{enumerate}
\item{$T_{N,\ell}$ is divisible by $x^N$.}
\item{The polynomial $S_{N+\ell}-T_{N,\ell}$ has degree at most $N-1$. That is, if $T_{N,\ell}$ is expanded in shadow Hermite polynomials, then the coefficients of $S_N,S_{N+1},\cdots,S_{N+\ell-1}$ are all zero.}
\end{enumerate}
These polynomials, which we call \textit{truncated shadow Hermite polynomials} arise naturally in the multivariate theory, where it is often necessary to look at the coefficient of the shadow Hermite polynomial $S_m(x;q)$ in $T_{N,\ell}(x;q)$, where $m\le N-1$. By applying the operator $E(-D_q^2/(1+q),q^2)$, we see that this is also the coefficient of $x^m$ in the polynomial:

\[T_{N,\ell}'(x;q) = \sum_{k=0}^{\ell/2} \left[{N+\ell\atop 2k}\right]_q M_q(2k-1) H_{N+\ell-2k}(x;q)\]
If $N+\ell$ and $m$ have different parities, then this coefficient is trivially zero. Otherwise, we can let $N + \ell - m = 2p$. Thus this coefficient is:

\[ \sum_{k=0}^{\ell/2} \left[{N+\ell\atop 2k}\right]_q\left[{N+\ell-2k\atop 2p-2k}\right]_q M_q(2k-1)M_q(2p-2k-1) (-1)^{p-k} q^{(p-k)(p-k-1)} \]
\[=\left[{N+\ell\atop 2p}\right]_q\sum_{k=0}^{\ell/2} \left[{2p\atop 2p-2k}\right]_q M_q(2k-1)M_q(2p-2k-1) (-1)^{p-k} q^{(p-k)(p-k-1)} \]
\[= \left[{N+\ell\atop 2p}\right]_q M_q(2p-1) \sum_{k=p-\ell/2}^{p} \left[{p\atop k}\right]_{q^2}  (-1)^{k} q^{k(k-1)}\]

By the $q$-binomial theorem, however,

\[\sum_{r = 0}^s q^{r(r-1)} \left[{n\atop r}\right]_{q^2} (-1)^r = (-1)^s q^{s(s-1)}\left[{n-1\atop s}\right]_{q^2}\]

Therefore, the above coefficient is just:
\[(-1)^{p-\lfloor \ell/2\rfloor} q^{(p-\lfloor \ell/2\rfloor-1)(p-\lfloor \ell/2\rfloor-2)} \left[{N+\ell\atop 2p}\right]_q\left[{p-1\atop p-\lfloor \ell/2\rfloor - 1}\right]_{q^2} M_q(2p-1)\]

Therefore:

\[T_{N,\ell} = S_{N+\ell} + \sum_{p=\frac{\ell+1}{2}}^{\frac{N+\ell}{2}} (-1)^{p-\lfloor \frac\ell2\rfloor} q^{(p-\lfloor \frac\ell2\rfloor-1)(p-\lfloor \frac\ell2\rfloor-2)} \left[{N+\ell\atop 2p}\right]_q\left[{p-1\atop p-\lfloor \ell/2\rfloor - 1}\right]_{q^2} M_q(2p-1)S_{N+\ell-2p}\]

\section{Multivariate Orthogonal Polynomials}

Let $\bm{x} = (x_1,...,x_N)$. If $M$ is an arbitrary linear functional from $\qq[x]$ to $\qq$, let $M_i$ be the linear map from $\qq[x_1,...,x_i,...,x_N]$ to $\qq[x_1,...,x_{i-1},x_{i+1},...,x_N]$ obtained by applying $M$ to the $i$th variable. In other words, if $p$ is a polynomial in $N$ variables, we can expand it in powers of $x_i$, as follows:
$$p(x) = \sum_{n} r_n(x_1,...,x_{i-1},x_{i+1},...,x_N)x_i^n$$
Then:
$$(M_i(p))(x_1,...,x_{i-1},x_{i+1},...,x_N) = \sum_{n} r_n(x_1,...,x_{i-1},x_{i+1},...,x_N)(M(x^n))$$
Thus, we can define linear functionals on polynomials in $N$ variables by successively applying $M_i$, for $i=1,2,\cdots,N$:
\begin{defi} For any linear functional $M:\qq[x]\to \qq$,
$$M^{(0)} = M_1M_2\cdots M_N = M_NM_{N-1}\cdots M_1$$
$$M^{(2)}(p) = M^{(0)}\left(p \prod_{1\le i<j\le N} (x_i-x_j)^2\right) = M^{(0)}(p\vv^2)$$
\end{defi}
(For shorthand, we use the symbol $\mathsf{V}$ to denote $\displaystyle\prod_{1\le i<j\le N} (x_i-x_j)$.) This definition is useful because of the following lemma:

\begin{lem} a) Suppose $f_\ell$ and $g_\ell$ are polynomials in $\qq[x]$, for $1\le \ell\le N$. Then: $$M^{(0)}(\det [f_j(x_k)]_{j,k=1}^N\det [g_\ell(x_k)]_{\ell,k=1}^N) = N!\det [M(f_jg_\ell)]_{j,\ell=1}^N$$

b)  Suppose $f_\ell$ and $g_\ell$ are polynomials in $\mathbb{Q}(q)[x]$, for $1\le \ell\le N$, and $\mathrm{CT}(h)$ be the constant term of the Laurent polynomial $h$. Then: $$\mathrm{CT}\left(\det [f_j(x_k)]_{j,k=1}^N\det [ g_\ell(x_k^{-1})]_{\ell,k=1}^N \right) = N! \det \left[\mathrm{CT}\left( f_j(x) g_\ell(x^{-1})\right) \right]_{j,\ell=1}^N$$

\end{lem}
\begin{proof}
Let $\rho_s(\sigma)$ denote the sign of $\sigma$, for $\sigma\in S_N$.

$$M^{(0)}(\det [f_j(x_k)]_{j,k=1}^N\det [g_\ell(x_k)]_{\ell,k=1}^N) = \sum_{\sigma,\tau\in S_N} \rho_s(\sigma)\rho_s(\tau)\prod_{j=1}^N M(f_{\sigma(j)}g_{\tau(j)})$$

$$=\sum_{\tau,\gamma\in S_N} \rho_s(\gamma) \prod_{j=1}^N M(f_{\gamma(\tau(j))}g_{\tau(j)}) = N!\det [M(f_jg_\ell)]_{j,\ell=1}^N$$ The proof of b) is very similar to that of a).
\end{proof}

\subsection{Schur Polynomials} We can define a multivariate analog of the basis of monomials, known as \textit{Schur polynomials}. These are $N$-variable polynomials indexed by partitions $\kappa = (\kappa_1,\kappa_2,\cdots,\kappa_N)$, defined by a determinant:
$$s_\kappa(x_1,\cdots,x_N) = \frac{\det (x_i^{\kappa_j + N - j})_{i,j=1}^N}{\det (x_i^{j-1})_{i,j=1}^N}$$
For example, when $\kappa = (1,0,\cdots,0)$, $s_1(x):=s_\kappa(x) = x_1+x_2+\cdots+x_N$. The Schur polynomials are a basis for the set of symmetric polynomials in $N$ variables. More interestingly, these polynomials obey an orthogonality relation:

\begin{lem}The following holds true, when $x_j := \exp i\theta_j$:
$$\frac{1}{(2\pi)^N}\int_{[0,2\pi]^N} s_\kappa(x_1,\cdots,x_N)s_\lambda(x_1^{-1},\cdots,x_N^{-1}) \prod_{1\le i<j\le N} |x_i -x_j|^2 d^N\theta = N!\delta_{\kappa;\lambda}$$
\end{lem}

\begin{proof}The left hand side is just the constant term of the integrand, which evaluates to: $$\det (x_i^{\kappa_j + N - j})\det (x_i^{-(\lambda_j + N - j)})$$ by the Vandermonde determinant formula. The rest follows from Lemma 5b.
\end{proof}

In addition, for any sequence $\{P_n\}$ of orthogonal polynomials in $\qq[x]$ we can define a family of multivariate symmetric orthogonal polynomials in $\qq[x_1,x_2,...x_N]$:

\begin{lem}
Let $\{f_n\}_{n=0}^\infty$ be monic polynomials in $\qq[x]$ such that $\deg f_n = n$, and let $M:\qq[x]\to\qq$ be a linear functional such that $M(f_nf_m) = 0$ when $n\ne m$.

a) Define $$F_\kappa(x_1,\cdots,x_N) = \frac{\det (f_{\kappa_j + N - j}(x_i))_{i,j=1}^N}{\det (x_i^{j-1})_{i,j=1}^N}$$
Then $M^{(2)}(F_\kappa F_\lambda) = 0$, and: $\kappa\ne\lambda$, and: $$M^{(2)}(F_\kappa F_\kappa) = N!\prod_{i} M(f_{\kappa_i + N - i}^2)$$

b) The expansion of $F_\kappa$ in Schur polynomials contains only those partitions $\lambda$ whose Young diagram is contained in $\kappa$, and its leading term is $1$. In particular, $\{F_\kappa\}$ is a basis for the vector space of real symmetric polynomials.
\end{lem}
\begin{proof}
$$M^{(2)}(F_\kappa F_\lambda) = N!\det [M(f_{\kappa_j + N - j}f_{\lambda_\ell + N - \ell})]_{j,\ell=1}^N$$

If $\kappa=\lambda$, the matrix on the right is diagonal with entries $M(f_{\kappa_j + N - j}^2)$, and hence its determinant equals $\prod_{j} M(f_{\kappa_j + N - j}^2)$. Otherwise, the matrix has a zero column, so it's singular.

b) By Lemma 5 and Lemma 6,

$$[s_\lambda]F_\kappa = \det \left[\mathrm{CT}(x^{\lambda_j + N - j}f_{\kappa_\ell + N - \ell}(x^{-1})) \right]_{j,\ell=1}^N$$

If $\lambda_J + N - J > \kappa_J + N - J$ for some $J$, then the entries of the matrix on the right indexed by $\ell = N, N-1,...,J, j=0, 1, ..., J$ are all zero. Thus its determinant is $0$. So if $[s_\lambda]P_\kappa\ne 0$, $\lambda_j + N - j \le \kappa_j + N - j$ for all $j$, so the Young diagram of $\lambda$ is contained in $\kappa$. If $\lambda = \kappa$, the matrix on the right is the identity matrix, so its determinant is $1$.
\end{proof}

\subsection{Some more identities} Let $\bm{x} = (x_1,...,x_N)$ as before. This time, let $\{g_n\}$ be arbitrary monic polynomials of in $\qq[x]$ degree $n$ (not necessaily orthogonal), where $n = 0,1,2,3,\dots$. Define the corresponding multivariate polynomials:

$$G_\kappa(\bm x) = \frac{\det(G_{\kappa_i+N-i}(x_j)_{i,j=1}^N}{\det(x_i^{j-1})_{i,j=1}^N}$$
Consider the polynomial:
$$ D_{N,\ell}\vv := G_{(\ell)}(\bm x, y) \cdot \left(\vv\prod_{i=1}^N (y-x_i)\right) = \det \left[
\begin{matrix}
g_{N+\ell}(x_1) & g_{N+\ell}(x_2) & \cdots & g_{N+\ell}(x_N) & g_{N+\ell}(y) \\
g_{N-1}(x_1) & g_{N-1}(x_2) & \cdots & g_{N-1}(x_N) & g_{N-1}(y) \\
g_{N-2}(x_1) & g_{N-2}(x_2) & \cdots & g_{N-2}(x_N) & g_{N-2}(y) \\
g_{N-3}(x_1) & g_{N-3}(x_2) & \cdots & g_{N-3}(x_N) & g_{N-3}(y) \\
\vdots & \vdots & \ddots & \vdots & \vdots \\
g_{1}(x_1) & g_{1}(x_2) & \cdots & g_{1}(x_N) & g_{1}(y) \\
g_{0}(x_1) & g_{0}(x_2) & \cdots & g_{0}(x_N) & g_{0}(y) \\
\end{matrix}\right]$$

We can expand it in $y$, as follows:

$$D_{N,\ell}\vv = (-1)^{N}\vv g_{N+\ell}(y) + \vv\sum_{i=0}^{N-1} (-1)^{i}g_{i}(y)G_{(\ell+1,(1)^{N-1-i})}(\bm{x})$$
where $(1)^{N-1-i}$ means $1$ repeated $N-1-i$ times. Dividing this equation by $\vv$ and setting $\ell = 0$, we get:

$$\prod_{i=1}^N (y-x_i) = (-1)^{N} g_{N}(y) + \sum_{i=0}^{N-1} (-1)^{i}g_{i}(y)G_{((1)^{N-i})}(\bm{x})$$

This latter formula is especially useful because it leads to a $q$-generalization of Wigner's famous level density formula.

\begin{thm}Let $p\in \qq[x]$, $M$ a linear functional and $\{f_i\}$ monic polynomials of degree $i$, $i=0,1,2,...$ such that $M(f_if_j) = 0$ when $i\ne j$. Furthermore, suppose that $M(f_i^2)\ne 0$ for every $i$. Then, if $k$ is an integer between $1$ and $N$,
$$\frac{M^{(2)}(p(x_1)+p(x_2)+\cdots+p(x_N))}{M^{(2)}(1)} = \sum_{j=0}^{N-1} \frac{M(pf_j^2)}{M(f_j^2)}$$
\end{thm}
\begin{proof} The proof follows from the identity:
$$M^{(2)}(p(x_1)+p(x_2)+\cdots+p(x_N)) = NM^{(0)}\left(p(x_1)\det \left(\sum_{k=0}^{N-1} f_k(x_i)f_k(x_j)\right)_{i,j=1}^N\right)$$
and the Laplace expansion formula.
\end{proof}

Likewise, we can divide by $\vv$ and set $\bm x = \bm 0 = (0,0,...,0)$ We get the following:

$$ y^N G_{(\ell)}(\bm 0, y) = (-1)^{N} g_{N+\ell}(y) + \sum_{i=0}^{N-1} (-1)^{i}g_{i}(y)G_{(\ell+1,(1)^{N-1-i})}(\bm{0})$$

In other words, if:
$$g_{N+\ell}(y) = \sum_{j=0}^{N+\ell} a_j y^j$$
then
$$ y^N G_{(\ell)}(\bm 0, y) = (-1)^N \sum_{j=N}^{N+\ell} a_j y^j$$

This implies the following theorem:

\begin{thm}
Let $p\in\qq[x]$ be the unique polynomial divisible by $y^N$ such that $g_{N+\ell} - p$ has degree at most $N-1$. Then $G_{(\ell+1,(1)^{i})}(\bm{0})$ is $(-1)^{N-1-i}$ times the coefficient of $g_{N-1-i}$ in the expansion of $p$ in $g_j$'s.
\end{thm}

\begin{rem}
If $g_n(x) = (x+1)^n$, the coefficient of $s_\kappa$ in $G_\lambda$ is the multivariate binomial coefficient $\displaystyle \binom{\lambda}{\kappa}$ defined in [2] and extended in [3].
\end{rem}

We now have everything we need to formalize the notion of a $q$-analog of the Gaussian unitary ensemble. We do this by setting $g_n = H_n(x;q)$ (the Hermite polynomials) or $g_n = S_n(x;q)$ (the shadow-Hermite polynomials) depending on context. We call the corresponding multivariate polynomials $\Eta_\kappa(\bm x;q)$ and $\Sigma_\kappa(\bm x;q)$ respectively. The inner product that orthogonalizes the Hermite polynomials is of course $f,g\mapsto \intn f(x)g(x)\,dx$, which we saw in Section 3. We can thus define the multivariate integral:

$$ \int_{q\dgue_N} f(\bm x) \, d\bm x := L^{(2)}(f)$$

By Lemma 6, we have:

$$ \int_{q\dgue_N} \Eta_\kappa(\bm x;q)\Eta_\lambda(\bm x;q) \, d\bm x := N!q^{{\sum_i \binom{\kappa_i + N - i}{2}}} \prod_i [\kappa_i+N-i]_q^!\delta_{\kappa;\lambda}$$

We can expand any multivariate symmetric polynomial $f$ in multivariate Hermite polynomials (since these form a basis):

$$f(\bm x) = \sum_{\kappa} c_\kappa \Eta_\kappa(\bm x;q)$$

It follows that:

$$\int_{q\dgue_N} f(\bm x)\,dx = c_\emptyset + \sum_{\kappa\ne\emptyset} c_\kappa \int_{q\dgue_N}\Eta_\kappa(\bm x;q)d\bm x = c_\emptyset$$

However, we can arrive at this constant by a more formal approach. We can expand $f$ in Schur polynomials:

$$f(\bm x) = \sum_{\kappa} c_\kappa' s_\kappa(\bm x;q)$$

and consider the polynomial $f'$ obtained by replacing the Schur polynomials with the shadow Hermite polynomials:

$$f'(\bm x) = \sum_{\kappa} c_\kappa' \Sigma_\kappa(\bm x;q)$$

We claim that $c_\emptyset$ is the constant term of $f'$ when expanded in Schur polynomials. By linearity, this is equivalent to the following theorem:

\begin{thm}
$$\int_{q\dgue_N} s_\kappa(\bm x)\,dx = \Sigma_\kappa(\bm 0)$$

for every partition $\kappa$.
\end{thm}
\begin{proof}
Let $d_{\kappa,\lambda}$ be constants such that:
$$\Eta_\kappa(\bm x) = \sum_{\lambda} d_{\kappa,\lambda}s_\lambda(\bm x)$$
Then:
\begin{align*}\Eta_\kappa(\bm x)\vv(\bm x) &= \sum_{\lambda} d_{\kappa,\lambda}s_\lambda(\bm x)\vv(\bm x) \\
&= \sum_{\lambda} d_{\kappa,\lambda} \det[x_i^{\lambda_j + N - j}]_{i,j=1}^N\end{align*}

Applying the operator $\prod_i e({_iD_q}^2/(1+q);q^2)$, we get:

\begin{align*}s_\kappa(\bm x)\vv(\bm x) &= \sum_{\lambda} d_{\kappa,\lambda}\left(\prod_i e({_iD_q}^2/(1+q);q^2)\right) s_\lambda(\bm x)\vv(\bm x) \\
&= \sum_{\lambda} d_{\kappa,\lambda} \left(\prod_i e({_iD_q}^2/(1+q);q^2)\right) \det[x_i^{\lambda_j + N - j}]_{i,j=1}^N \\
&= \sum_{\lambda} d_{\kappa,\lambda}  \det[S_{\lambda_j + N - j}(x_i)]_{i,j=1}^N \\
&= \sum_{\lambda} d_{\kappa,\lambda} \Sigma_\lambda(\bm x)\vv(\bm x)\end{align*}

Canceling $\vv(\bm x)$ and setting $\bm x=\bm 0$ gives the desired result.\qedhere
\end{proof}
\begin{cor}
Suppose $f(\bm x) = \sum_\lambda b_\lambda s_\lambda(\bm x)$. Then:
$$\int_{q\dgue_N} f(\bm x)d\bm x = \sum_\lambda b_\lambda \Sigma_\lambda(\bm 0)\textrm{ .}$$
\end{cor}

\section{The Moments of the $q$-GUE}

With this formula at hand we can compute expectations of a number of interesting statistics over the $q$-analog of the GUE. For example, let $\ell$ and $m$ be nonnegative integers and let $\mu\vdash 2m$ be a partition of $2m$, such that $\mu_1 = \ell+1$, $\mu_2 = \mu_3 = \cdots = \mu_{j} = 1$, and $\mu_{j+1} = \mu_{j+2} = \cdots = 0$ where $j+1+\ell = 2m$. Then, by Theorem 3 we have:

$$\int_{q\dgue_N} s_\mu(\bm x)\,dx = \Sigma_\mu(\bm 0)$$

However, we can use Theorem 2 to find  an explicit formula for $\Sigma_\mu(\bm 0)$. Namely, if $p\in\qq[x]$ is the unique polynomial divisible by $y^N$ such that $S_{N+\ell} - p$ has degree at most $N-1$. Then $\Sigma_{(\ell+1,(1)^{i})}(\bm{0})$ is $(-1)^{N-1-i}$ times the coefficient of $S_{N-1-i}$ in the expansion of $p$ in $S_j$'s. But we know that the polynomial $p$ is precisely the truncated shadow Hermite polynomial $T_{N,\ell}$, which is also given by:

\[T_{N,\ell} = S_{N+\ell} + \sum_{p=\frac{\ell+1}{2}}^{\frac{N+\ell}{2}} (-1)^{p-\lfloor \frac\ell2\rfloor} q^{(p-\lfloor \frac\ell2\rfloor-1)(p-\lfloor \frac\ell2\rfloor-2)} \left[{N+\ell\atop 2p}\right]_q\left[{p-1\atop p-\lfloor \ell/2\rfloor - 1}\right]_{q^2} M_q(2p-1)S_{N+\ell-2p}\]

To find the coefficient of $S_{N-1-i}$, we simply let $N+\ell-2p = N-1-i$, or $2p = \ell + 1 + i = 2m$. Namely,

\begin{thm} Let $\mu = (\ell+1,1,...,1,0,...,0)$ be a partition of $2m$. Then:
\[\int_{q\dgue_N} s_\mu(\bm x)\,d\bm x = (-1)^{m-\lfloor \frac\ell2\rfloor} q^{(m-\lfloor \frac\ell2\rfloor-1)(m-\lfloor \frac\ell2\rfloor-2)} \left[{N+\ell\atop 2m}\right]_q\left[{m-1\atop\lfloor \ell/2\rfloor}\right]_{q^2} M_q(2m-1)\]
\end{thm}

Define:
$$\sigma_{m,t}(\bm x) = s_{\mu_1}(\bm x) - s_{\mu_2}(\bm x)$$
where $\mu_1 = (2t,1,...,1,0,...,0)$ and $\mu_2 = (2t+1,1,...,1,0,...,0)$ are partitions of $2m$. Among other things, we have, by the Murnaghan-Nakayama rule,

$$p_{2m} = \sum_{t \ge 0} \sigma_{m,t}$$

Likewise,
\begin{align*}\int_{q\dgue_N}\sigma_{m,t}(\bm x)\,d\bm x &= (-1)^{m-t} q^{(m-t-1)(m-t-2)}M_q(2m-1)\left[{m-1\atop t}\right]_{q^2}\left(\left[{N+2t\atop 2m}\right]_q-\left[{N+2t+1\atop 2m}\right]_q\right) \\
&= (-1)^{m-t+1} q^{(m-t-1)(m-t-2)+(N+2t+1-2m)}M_q(2m-1)\left[{m-1\atop t}\right]_{q^2}\left[{N+2t\atop 2m-1}\right]_q \end{align*}

Hence,

\begin{align*}\int_{q\dgue_N} p_{2m}(\bm x)d\bm x &= M_q(2m-1)q^{N+1-2m}\sum_{t\ge 0} (-1)^{m-t+1} q^{(m-t-1)(m-t-2)+2t}\left[{m-1\atop t}\right]_{q^2}\left[{N+2t\atop 2m-1}\right]_q
\\ &= M_q(2m-1)q^{N+m^2-5m+3}\sum_{t\ge 0} (-1)^{m-t+1} q^{t^2 + (5-2m)t}\left[{m-1\atop t}\right]_{q^2}\left[{N+2t\atop 2m-1}\right]_q\end{align*}

However, by Theorem 1 the LHS is also equal to:

\begin{align*}\int_{q\dgue_N} p_{2m}(\bm x)d\bm x &= \left(\int_{q\dgue_N}d\bm x\right) \sum_{j=0}^{N-1} \frac{\intn x^{2m}H_j(x)^2\,dx}{\intn H_j(x)^2\,dx}
\\ &= \Sigma_\emptyset(\bm 0) \sum_{j=0}^{N-1} \frac{\intn x^{2m}H_j(x)^2\,dx}{q^{j(j-1)/2}[j]_q^!}
\\ &= \sum_{j=0}^{N-1} \frac{q^{-j(j-1)/2}}{[j]^!_q} \intn x^{2m}H_j(x)^2\,dx\end{align*}

Therefore, if $s$ is a positive integer,

$$
\frac{1}{q^{s(s+1)/2}[s+1]_q^!}\intn x^{2m} H_s(x)^2\,dx = \left(\int_{q\dgue_{s+1}} p_{2m}(\bm x)d\bm x - \int_{q\dgue_s} p_{2m}(\bm x)d\bm x\right)
$$ $$= M_q(2m-1)q^{m^2-5m+3}\sum_{t\ge 0} (-1)^{m-t+1} q^{t^2 + (5-2m)t + s}\left[{m-1\atop t}\right]_{q^2}\left(q\left[{s+1+2t\atop 2m-1}\right]_q-\left[{s+2t\atop 2m-1}\right]_q\right)
$$
$$ = M_q(2m-1)q^{m^2-5m+3}\sum_{t\ge 0} (-1)^{m-t+1} q^{t^2 + (5-2m)t + s}\left[{m-1\atop t}\right]_{q^2}\left[{s+2t\atop 2m-1}\right]_q\left(\frac{q(1-q^{s+1-2t})}{1 - q^{s+2+2t-m}} - 1\right)$$

This last equality is remarkable enough to be stated as a theorem:

\begin{thm}

$$\frac{1}{q^{s(s+1)/2}[s+1]_q^!}\intn x^{2m} H_s(x)^2\,dx$$
$$ = M_q(2m-1)q^{m^2-5m+3}\sum_{t\ge 0} (-1)^{m-t+1} q^{t^2 + (5-2m)t + s}\left[{m-1\atop t}\right]_{q^2}\left[{s+2t\atop 2m-1}\right]_q\left(\frac{q(1-q^{s+1-2t})}{1 - q^{s+2+2t-m}} - 1\right)$$

\end{thm}

\

\section{Future Work}

The right hand side of Theorem 5 looks a lot like the $q$-analog of the Harer-Zagier formula conjectured and eventually proven by Wimberley and Morales in [1]:

$$\frac{1}{q^{s(s-1)/2}[s]_q^!}\intn x^{2m} H_s(x)^2\,dx = \sum_{k\ge 0} q^{n(s-k)+k(k-1)/2} M_q(2n-1)\left[{s\atop k}\right]_{q}\left[{n\atop k}\right]_{q}\prod_{i=1}^k (1+q^{n+i})$$

We believe that this last identity can be proven another way, from Theorem 5 and the $q$-Zeilberger algorithm.

In a forthcoming paper, we plan to look at the $q$-analog of the Wishart ensemble and its corresponding orthogonal polynomials, the Laguerre polynomials, in the same way, thus developing new identities.

\section{Acknowledgements}

The author withes to thank Alan Edelman, Richard Stanley, Michael La Croix, Maria Monks, and Pavel Etingof for their support and helpful discussions and insights.

\section{Bibliography}

1. Maxim Wimberley, \textit{Rook Placements on Young Diagrams: Towards a q-Analogue of
the Harer-Zagier Formula}

\verb"https://math.mit.edu/news/summer/SPURprojects/2012Wimberley.pdf"

2. Macdonald, I. G. \emph{Symmetric functions and Hall polynomials}. Second edition. Oxford Mathematical Monographs. Oxford Science Publications. The Clarendon Press, Oxford University Press, New York, 1995. x+475 pp. ISBN 0-19-853489-2 MR 96h:05207

3. Ioana Dumitriu, \textit{Eigenvalue Statistics for $\beta$-Ensembles}, thesis.

\verb"http://www.math.washington.edu/~dumitriu/main.pdf"

4. Forrester, Peter J. \textit{Random matrices, log-gases and the Calogero-Sutherland model}. Quantum Many-Body Problems and Representation Theory, 97--181, The Mathematical Society of Japan, Tokyo, Japan, 1998. doi:10.2969/msjmemoirs/00101C020. 

5. Stanley, Richard P. (1989), ``Some combinatorial properties of Jack symmetric functions'', \emph{Advances in Mathematics} {\bf 77} (1): 76–115, doi:10.1016/0001-8708(89)90015-7, MR 1014073.

6. M. La Croix, \emph{The combinatorics of the Jack parameter and the genus series for topological maps}, Ph.D. at the University of Waterloo, 2009

7. R. Koekoek and R. Swarttouw, \textit{The Askey-scheme of hypergeometric orthogonal polynomials and its q-analogue} 98-17, Delft University of Technology, Faculty of Information Technology and Systems, Department of Technical Mathematics and Informatics

8. R. Stanley, P. Hanlon and J. Stembridge, \textit{Some Combinatorial Aspects of the Spectra of Normally Distributed Random Matrices} Contemp. Math. \textbf{138} (1992) 151-174.

9. C. Balderrama, P. Graczyk and W. Urbina, \textit{A formula for polynomials of Hermitian matrix argument}, Bulletin des Sciences Math\' ematiques \textbf{129}, Issue 6 (2005), 486-500.

\end{document}